\RequirePackage{fix-cm}
\documentclass[smallextended]{svjour3_2}       
\smartqed  
\usepackage{graphicx}
\usepackage{amsmath,tikz,color,calc,amssymb}
\usepackage{microtype}
\usepackage[bitstream-charter]{mathdesign}

\usepackage{oldgerm}
\DeclareFontFamily{U}{yswab}{}	
\DeclareFontShape{U}{yswab}{m}{n}{<-> yswab}{}

\usepackage{bm}
\usepackage[final]{showkeys}

\numberwithin{equation}{section}
\usepackage[hmargin=2.5cm,vmargin={2cm,3.5cm}]{geometry}
\allowdisplaybreaks

\newcommand{\be}{\begin{equation}}
\newcommand{\ee}{\end{equation}}

\def\bE{\mathbb{E}}
\def\bN{\mathbb{N}}
\def\bP{\mathbb{P}}

\def\bR{\mathbb{R}}
\def\bZ{\mathbb{Z}}

 \def\Z{\bZ} 
 
\def\R{\bR}
\def\N{\bN}

\def\e{\varepsilon}
 
\def\E{\bE}
\def\P{\bP} 

\definecolor{darkgreen}{rgb}{0.0,0.5,0.0}
\definecolor{darkblue}{rgb}{0.0,0.0,0.3}
\definecolor{nicosred}{rgb}{0.65,0.1,0.1}
\definecolor{light-gray}{gray}{0.7}
\allowdisplaybreaks[1]

\RequirePackage{datetime} 

\begin{document}

\title{A stylized model for wealth distribution
}


\author{Bertram D\"uring \and Nicos Georgiou \and Enrico Scalas}


\institute{Bertram D\"uring \at
            Mathematics Institute, University of Warwick, UK \\
              \email{bertram.during@warwick.ac.uk}           \\
           \and
           Nicos Georgiou \at
             Department of Mathematics, University of Sussex, UK \\
\email{n.georgiou@sussex.ac.uk}  \\
\and
Enrico Scalas \at 
Department of Mathematics, University of Sussex, UK\\
\email{e.scalas@sussex.ac.uk}}

\date{}

\maketitle

\begin{abstract}
The recent book by T. Piketty (Capital in the Twenty-First Century) promoted the important issue of wealth inequality. In the last twenty years, physicists and mathematicians developed models to derive the wealth distribution using discrete and continuous stochastic processes (random exchange models) as well as related Boltzmann-type kinetic equations. In this literature, the usual concept of equilibrium in Economics is either replaced or completed by statistical equilibrium.

In order to illustrate this activity with a concrete example, we present a stylised random exchange model for the distribution of wealth. We first discuss a fully discrete version (a Markov chain with finite state space). We then study its discrete-time continuous-state-space version and we prove the existence of the equilibrium distribution. Finally, we discuss the connection of these models with Boltzmann-like kinetic equations for the marginal distribution of wealth. This paper shows in practice how it is possible to start from a finitary description and connect it to continuous models following Boltzmann's original research program.
\end{abstract}

\section{Introduction}

The recent book by T. Piketty \cite{Piketty} shifted the general attention as well as the attention of economists towards the important issue of wealth inequality. The question ``Why is there wealth inequality?" 
has attracted the attention of a diverse set of researchers, including economists, physicists and mathematicians. Particularly, 
during the last twenty years, physicists and mathematicians developed models to theoretically derive the wealth distribution using tools of statistical physics and probability theory: 
discrete and continuous stochastic processes (random exchange models) as well as related Boltzmann-type kinetic equations. In this framework, 
the usual concept of equilibrium in Economics is complemented or replaced by statistical equilibrium \cite{garibaldiscalas}. 

The original work of Pareto concerned the distribution of income \cite{Pareto}. Pareto observed a skewed distribution with power-law tail. However, he also dealt with the distribution of wealth, for which he wrote:
\begin{quote}
La r\'epartition de la richesse peut d\'ependre de la nature des hommes dont se compose la societ\'e, de l'organisation de celle-ci, et aussi, en partie, du {\em hasard} (les {\em conjonctures} de Lassalle), [...]
\end{quote}
\begin{quote}
The distribution of wealth can depend on the nature of those who make up society, on the social organization and, also, in part, on {\em chance}, (the {\em conjunctures} of Lassalle), [...]
\end{quote}
More recently, Champernowne \cite{Champernowne}, Simon \cite{Simon},
Wold and Whittle \cite{Wold} as well as Mandelbrot \cite{Mandelbrot}
used random processes to derive distributions for income and
wealth. Starting from the late 1980s and publishing in the
sociological literature, Angle introduced the so-called inequality
process, a continuous-space discrete time Markov chain for the
distribution of wealth based on the surplus theory of social
stratification \cite{Angle}. However, the interest of physicists and
mathematicians was triggered by a paper written by Dr\v{a}gulescu and
Yakovenko \cite{Yakovenko} in 2000 and explicitly relating random
exchange models with statistical physics. Among other things, they
discussed a simple random exchange model already published in Italian
by Bennati \cite{Bennati}. An exact solution of that model was
published in \cite{Donadio}. Lux wrote an early review of the
statistical physics literature up to 2005 \cite{Lux}. An extensive
review was written by Chakrabarti and Chackrabarti in 2010
\cite{Chakrabarti}. Boltzmann-like kinetic equations for the marginal
distribution of wealth were studied in \cite{CPT05} and several other
works, we refer to the review article \cite{DMT09} and the book
\cite{ParTosBook}, and the references therein. 

We will focus on the essentials of random modelling for wealth distributions and we will explicitly show how continuous-space Markov chains 
can be derived from discrete-space (actually finite-space) chains. We will then focus on the stability properties of these chains and, 
finally, we will review the mathematical literature on kinetic equations while studying the kinetic equation related to the Markov chains. 
In doing so, we will deal with a stylized model for the time evolution of wealth (a {\em stock}) and not of income (a {\em flow}).

Distributional problems in Economics can be presented in a rather general form. Assume one has $N$ economic agents, each one
endowed with his/her stock (for instance wealth) $w_i \geq 0$. Let 
$W = \sum_{i=1}^N w_i $ 
be the total wealth of the set of agents. Consider the random
variable $W_i$, i.e. the stock of agent $i$. One is interested in the distribution of the vector $(W_1, \ldots, W_N)$
as well as in the marginal distribution $W_1$ if all agents are on a par (exchangeable).
The transformation
\begin{equation}
\label{fraction}
X_i = \frac{W_i}{W},
\end{equation}
normalises the total wealth of the system to be equal to one since 
\begin{equation}
\label{sum}
\sum_{i=1}^N X_i = 1
\end{equation}
and the vector $(X_1, \ldots, X_N)$ is a finite random partition of the interval $(0,1)$. The $X_i$s are called {\em spacings} of the partition.

The following remarks are useful and justify this simplified modelling of wealth distribution.

\begin{enumerate}

	\item If the stock $w_i$ represents wealth, it can be negative due to indebtedness. In this 		case, one can always shift the wealth to non-negative values by
		subtracting the negative wealth with largest absolute value.

	\item A mass partition is an infinite sequence $\mathbf{s} = (s_1, s_2, \ldots)$ such that 		$s_1 \geq s_2 \geq \ldots \geq 0$ and $\sum_{i=1}^\infty s_i \leq 1$.

	\item Finite random interval partitions can be 
	mapped into mass partitions, just by ranking the 
	spacings and adding an infinite sequence of $0$s.

\end{enumerate}

The vector $\mathbf{X} = (X_1, \ldots, X_N)$ lives on the $N-1$ dimensional simplex $\Delta_{N-1}$, defined by

\begin{definition}[The simplex $\Delta_{N-1}$]
\begin{equation}
\label{simplex}
\Delta_{N-1} = \left \{ \mathbf{x} = (x_1, \ldots, x_N): \, x_i \geq 0 \, \text{ for all } i = 1, \ldots, N \text{ and } \sum_{i=1}^N x_i = 1 \right \}.
\end{equation}

\end{definition} 
There are two natural questions that immediately arise from defining such a model.
\begin{enumerate}

	\item Which is the distribution of the vector 
	$(X_1, \ldots, X_N)$ with $X_i$ given by (\ref{fraction}) at a given time?

	\item Which is the distribution of the random variable 
	$X_1$, the proportion of the wealth of a single individual?

\end{enumerate} 

One well-studied probabilistic example is to  set the vector $(W_1, \ldots, W_N)$ of i.i.d.~random variables such that $W_i \sim$ gamma$(\alpha_i, \lambda)$. Then 
$W = \sum_{i=1}^N W_i \sim$  gamma$\left(\sum_{i=1}^N \alpha_i,\lambda\right)$. 
In this case the mass function of $(X_1, \ldots, X_N)$ is the Dirichlet distribution, given by

\begin{equation}
\label{dirichlet}
f_{\mathbf{X}}(\mathbf{x}) = \frac{\Gamma(\alpha_1 + \cdots + \alpha_N)}{\Gamma(\alpha_1) \cdots \Gamma(\alpha_N)} x_1^{\alpha_1 -1} \cdots x_N^{\alpha_N -1}, \quad \mathbf{x} = (x_1, \ldots, x_N) \in \Delta_{N-1}.
\end{equation}
 
 We say $\mathbf{X} \sim$ Dir$_{N-1} (\alpha_1, \ldots, \alpha_N)$ and the parameters $\alpha_1, \ldots, \alpha_N $ are assumed strictly positive as they can be interpreted as the shapes of gamma random variables. A particular case is when 
 $\alpha_1 = \cdots = \alpha_n = \alpha$. Then the Dirichlet distribution is called symmetric. The symmetric Dirichlet distribution with $\alpha=1$ is uniform on the
simplex $\Delta_{N-1}$.

One can now answer the two questions above using the following proposition which we present in its simplest form.

\begin{proposition}
\label{uniformdirichlet}
Let $(W_1, \ldots, W_N)$ of i.i.d. random variables such that $W_i \sim \exp(1)$. Then 
$W = \sum_{i=1}^N W_i \sim$ {\em gamma}$(N,1)$. Define $X_i = W_i/W$, then the vector
$\mathbf{X} = (X_1, \ldots, X_N)$ has the uniform distribution on the simplex $\Delta_{N-1}$
and  one dimensional marginals $X_1 \sim$ {\em beta}$(1,N-1)$, namely
\begin{equation}
\label{betadistribution}
f_{X_1} (x) = \frac{(1-x)^{N-2}}{B(1,N-1)},
\end{equation}
where, for $a,b > 0$,
\begin{equation}
\label{betafunction}
B(a,b) = \frac{\Gamma(a) \Gamma(b)}{\Gamma(a+b)}.
\end{equation}
\end{proposition}
The proof of this proposition can be found in several textbooks of probability and statistics including Devroye's book \cite{Devroye}. Specifically, the part of proposition \ref{uniformdirichlet} concerning the uniform distribution is a corollary of Theorem 4.1 in \cite{Devroye}. Equation \eqref{betadistribution} is a direct consequence of the aggregation property of the Dirichlet distribution.

In this chapter, we define three related models that incorporate a stochastic time evolution for the agent wealth distribution. The models increase in mathematical complexity in the order they are presented. 

The first one is a discrete time-discrete (DD) space Markov chain with a P\'olya limiting invariant distribution. We keep the dynamics as simple as possible so in fact the invariant distribution will be uniform (not a generic P\'olya distribution) but the ideas and techniques are the same for more complicated versions. The Markov chain of the DD model is then generalised to a discrete time-continuous space  (DC) Markov chain. The extension is natural in the sense that the dynamics, irreducibility and the invariant distribution of the DC model can be viewed as limiting case of the DD model. In the process, we effectively prove that Monte Carlo algorithms will approximate the DC model well. Finally, we present a continuous-continuous space (CC) model for which the temporal evolution of the (random) wealth of a single individual is governed by a Boltzmann type equation. 

\section{Random dynamics on the simplex}


In order to define our simple models, we first introduce two types of moves on the simplex.

\begin{definition}[Coagulation]
By coagulation, we denote the aggregation of the stocks of two or more agents into a single stock. This can happen in mergers,
acquisitions and so on.
\end{definition}

\begin{definition}[Fragmentation]
By fragmentation, we denote the division of the stock of one agent into two or more stocks. This can happen in inheritance, failure and so on.
\end{definition}

\subsection{Discrete time - continuous space model: Coagulation-fragmentation dynamics}
\label{sec:DC1}

Before introducing the DD model, let us define the main object of our study: The DC model.

At each event time, the state of the process $\mathbf{X} \in \Delta_{N-1}$ changes according to a composition of one coagulation and one fragmentation step. 

To be precise, let $\mathbf{X} =\mathbf{x}$ be the current value of the random variable $\mathbf{X}$. For any ordered pair of indices $ i, j$, $1 \le i , j \le N$, chosen uniformly at random, define the coagulation application 
$\text{coag}_{ij} (\mathbf{x}): \Delta_{N-1} \to \Delta_{N-2}$ by creating a new agent with
stock $x = x_i + x_j$ while the proportion of wealth for all others remain uncganged. 
Next enforce a random fragmentation application $\text{frag}(\mathbf{x}): \Delta_{N-2} \to \Delta_{N-1}$
that takes $x$ defined above and splits it into two parts as follows. Given $u \in (0,1)$ drawn from the uniform distribution $U[0,1]$,
set $x_i = u x$ and $x_j = (1-u)x$.  

The sequence of coagulation and fragmentation operators defines a time-homogeneous Markov chain on the simplex $\Delta_{N-1}$.
Let $\mathbf{x}(t) = (x_1(t), \ldots, x_i (t), \ldots, x_j (t), \ldots, x_N (t))$ be the state of the chain at time $t$ with $i$ and $j$ denoting the selected indices.
Then the state at time $t+1$ is \[\mathbf{x}(t+1) = (x_1(t+1), \ldots, x_i (t+1) = u(x_i(t)+x_j(t)), \ldots, x_j (t+1) = (1-u)(x_i (t) + x_j (t)), \ldots, x_N (t)).\]

The Markov kernel for this process is however degenerate because each step only affects a Lebesgue measure 0 of the simplex. To avoid this technical complication for the moment, we define the same dynamics on the discrete simplex and we then analyse the DC model.

\subsection{Discrete time - discrete space model}

 Let $N$ denote the number of categories (individuals) into which $n$ objects (coins or tokens) are classified \cite{garibaldiscalas}. 
In the frequency or statistical description of this system, a state is a list $\mathbf{n} = (n_1, \ldots, n_N)$ with
$\sum_{i=1}^N n_i = n$ which gives the number of objects belonging to each category. In this framework, a coagulation move is 
defined by picking up a pair of  ordered integers $i, j$ at random without replacement from $1\le \ldots \le N$ and creating a new category with $n_i+n_j$ objects. 
A fragmentation move takes this category and splits it into two new categories relabeled $i$ and $j$ where $n'_i$ is a uniform random integer between $0$ and $n_i+n_j $ and $n'_j=n_i+n_j - n'_i$. The state of the process at time $t \in \N_0$ is denoted by  $\mathbf{X}(t)$ and its state space is  the scaled integer simplex
	\[ 
	S_{N-1}^{(n)} = n\Delta_{N-1}\cap\Z^N = \left\{ \mathbf n = (n_1, n_2, \ldots, n_N) : 0 \le n_i \le n,\quad \sum_{i=1}^Nn_i = n, \quad n_i \in \N_0.\right\}.
	\]



Formally, with coagulation, we move from the state space $S_{N-1}^{(n)}$ to $S_{N-2}^{(n)}$ and then again with fragmentation, we come back to 
$S_{N-1}^{(n)}$. While it is interesting to actually study all stages of the procedure, we are  only interested in the aggregated wealth and therefore we can by-pass the 
intermediate state space by defining the process only on  $S_{N-1}^{(n)}$; it is straight forward to write down the transition probabilities for $\mathbf{X}(t)$
\begin{align}
\label{transition}
\mathbb{P}&\{\mathbf{X}(t+1)= \mathbf{n}'|\mathbf{X}(t)=\mathbf{n}\} \notag\\
&=\sum_{i,j : i \neq j}\left\{  \frac{1}{N} \frac{1}{N-1}\frac{1}{n_i+n_j+1} 
\delta_{n_i+n_j,n'_i+n'_j} \prod_{k \neq i,j} \delta_{n'_k,n_k} \right\}. 
\end{align}
The notation is shorthand and implies that we are adding over all ordered pairs $(i,j), i\neq j$ where the first coordinate indicates the index $i$ that was selected first. This model is symmetric in $\bf n$, $\bf n'$. 


The chain is \textit{time-homogeneous} as the transition \eqref{transition} is independent of the time parameter $t$. 
It is also \textit{aperiodic} since with positive probability, during each time step, the chain may 
coagulate and then fragment to the same state. To see this, consider any vector 
$(X_1, \ldots, X_N) = (x_1, \ldots, x_N)$ on the simplex. It must have at least one non-zero entry, say $x_1 >0$. Select index $i =1$ first (with probability $N^{-1}$), then select any other index $j$. 
After that fragment at precisely $x_1, x_j$ (with probability $1/(x_1 + x_j +1)> 0$). Finally, the chain is \textit{irreducible}, since from any point 
$\mathbf{X} = (x_1, \dots, x_N)$ the chain can move with positive probability to any of the neighbouring  $((x_1, \dots, x_N) \pm (e_i-e_j) )\cap S_{N-1}^{(n)}$, i.e. to any point in the simplex, that is $\ell^1$-distance 2 away from the current state. 
Therefore, we can conclude that the chain $\{\mathbf{X}(t)\}_{t \in \N_0}$ has a unique equilibrium distribution $\mathbf{\pi}$ which we identify in the next proposition.

\begin{proposition}
 The invariant distribution of this Markov chain $\mathbf{X}(t)$ is the uniform distribution on 
 $n\Delta_{N-1}\cap\Z^N$.
\end{proposition}

\begin{proof} 

Define 
\[ 
A_{i,j}(\mathbf n')= \left\{ \mathbf n : \mathbf n \stackrel{\text{coag-frag}_{i,j}}{\longrightarrow} \mathbf n'\right\}
\]
to be the set of all simplex elements $\mathbf n$ that map to $\mathbf n'$ via a coagulation-fragmentation procedure in the $i, j$ indices ($i$ selected before $j$). This set is never empty as it can always contain the vector $\bf n'$.
For fixed $(i,j)$ its cardinality is
\be 
\text{card}(A_{i,j}(\mathbf n')) = n'_i+n_j' +1= n_i+n_j +1.
\ee
Using this notation, we may re-write the transition probability in \eqref{transition} as 
\begin{align}
\mathbb{P}&\{\mathbf{X}(t+1)= \mathbf{n}'|\mathbf{X}(t)=\mathbf{n}\} \notag\\
&=\sum_{i,j : i \neq j}\left\{  \frac{1}{N} \frac{1}{N-1}\Bigg(\frac{1}{n'_i+n'_j +1}\Bigg)
\delta_{n_i+n_j,n'_i+n'_j} \prod_{k \neq i,j} \delta_{n'_k,n_k} \right\}\notag\\
&=\sum_{i,j : i \neq j}\left\{  \frac{1}{N} \frac{1}{N-1}\Bigg(\frac{1}{\text{card}(A_{i,j}(\mathbf n')) }\Bigg)
\delta_{n_i+n_j,n'_i+n'_j} \prod_{k \neq i,j} \delta_{n'_k,n_k} 1\!\!1\left\{\mathbf n \in A_{i,j}(\mathbf n')\right\}\right\}\notag\\
&\phantom{xxxxxxxxxxxxxxxnxnxnxnxxx} \text{since the $\delta$ product is equivalent to the last indicator,}\notag\\
&=\sum_{i,j : i \neq j}\left\{  \frac{1}{N} \frac{1}{N-1}\Bigg(\frac{1}{\text{card}(A_{i,j}(\mathbf n')) }\Bigg)1\!\!1\left\{\mathbf n \in A_{i,j}(\mathbf n')\right\} \right\}\label{eq:tr2}
\end{align}

Now fix a $\mathbf n'$ and add up all the transition probabilities in \eqref{eq:tr2} over 
$\mathbf n$. We get 
\begin{align*}
\sum_{\mathbf n}\sum_{i,j : i \neq j} &\left\{  \frac{1}{N} \frac{1}{N-1} 
\Bigg(\frac{1}{\text{card}(A_{i,j}(\mathbf n'))}\Bigg)1\!\!1\{\mathbf n \in A_{i,j}(\mathbf n')\} \right\}\notag\\
&=\sum_{i,j : i \neq j} \left\{  \frac{1}{N} \frac{1}{N-1} \Bigg(\frac{1}{\text{card}(A_{i,j}(\mathbf n'))}\Bigg)\sum_{\mathbf n}1\!\!1\{\mathbf n \in A_{i,j}(\mathbf n')\} \right\}\notag\\
&=\sum_{i,j : i \neq j} \left\{  \frac{1}{N} \frac{1}{N-1} \Bigg(\frac{1}{\text{card}(A_{i,j}(\mathbf n'))}\Bigg) \text{card}(A_{i,j}(\mathbf n'))\right\}\notag\\
&=1.\notag
\end{align*}
Therefore the transition matrix is doubly stochastic and in particular the invariant distribution must be uniform.
\end{proof}

\subsection{Convergence of the finite Markov chain as the overall wealth increases}

Reaching a similar conclusion in the case of the DC model is slightly more complicated. The difficulty is related to the fact that time is changing in discrete steps and the chain cannot explore the whole available state space because real numbers cannot be put in 1-to-1 correspondence with integers. How can we be sure that the Markov chain with continuous state space can explore its state space uniformly? We begin our analysis by studying the convergence of the finite state-space Markov chain to the continuous state-space Markov chain.

Let $\mathbf X^{(n)}$ be the DD Markov chain for wealth, when the wealth of the system is  and $n$ and let $\mathbf X^{(\infty)}$ be the chain for the DC model introduced in Section \ref{sec:DC1}. We scale the state space of each process $\mathbf X^{(n)}$ so that it is a subset of $\Delta_{N-1}$ by defining a new, coupled process 
\[\mathbf Y^{(n)} = n^{-1} \mathbf X^{(n)}. \]
The state space for the process $\mathbf Y^{(n)}$ is the simplex 
\[
\Delta_{N-1}(n)= \{ ( q_1, \ldots, q_d): 0 \le q_i \le 1,\, q_1 + \ldots + q_d = 1, \, n q_i \in \N_0 \} \subset \Delta_{N-1}.
\]
It can be considered as partition of $\Delta_{N-1}$ with mesh $n^{-1}$, i.e. inversely proportional of the total wealth. 

In this section we first prove weak convergence of the one-dimensional marginals 
\[ \mathbf Y^{(n)}_k \stackrel{n\to \infty}{\Longrightarrow} \mathbf X^{(\infty)}_k, \text{for all } k \in \N \] 
and then prove the existence of a unique invariant distribution for $X^{(\infty)}$  (the DC model) that we identify as the uniform distribution on $\Delta_{N-1}$.

Let $\mu_0^{(n)}$ the initial distribution of 
$Y_0^{(n)}$ and $\mu_0^{(\infty)}$ the initial distribution of  $X^{\infty}_0$.

\begin{proposition} Assume the weak convergence $\mu_0^{(n)} \Longrightarrow \mu^{(\infty)}_0$ as $n \to \infty$. Then for each $k \in \N$ we have weak convergence of the one-dimensional 
marginals 
\[ Y_k^{(n)} \Longrightarrow X_k^{(\infty)}\, \text{ as } n \to \infty.\]
\end{proposition}

\begin{proof} 

	We first show this for $k= 1$ and then show it in general with an inductive argument. 
Let $f$ be a bounded continuous function on $\Delta_{N-1}$. Let $U$ be a uniform random variable on $[0,1]$
and define the bounded and continuous
$F_{i,j}: \Delta_{N-1} \to \R$ by 
	\[
		F_{i,j}(x_1, \ldots, x_d) = \E^{U}(f (x_1, \ldots, U(x_i + x_j), x_{i+1},  \ldots , (1 -U)(x_i + x_j),x_{j+1}, \ldots, x_d)).
	\]
Pick an $\e>0$. By compactness, we can find a $\delta = \delta(\e)>0$ such that whenever $\|x - y\|_1 < \delta$ we have that 
\[ \sup_{\{i,j\}}| F_{ij}(x) - F_{ij}(y)| + |f(x) - f(y)| < \e.\]
From this relation, choose $n$ large enough so that the discrete simplex $\Delta_{N-1}(n)$ is fine enough, namely two neighboring points $x^{(n)}, y^{(n)}$ satisfy $\|x^{(n)} - y^{(n)}\|_1 < \delta$. In particular, this implies that $n> 2\delta^{-1}$.

The function $F_{i,j}$ evaluated on the partition points is 
\begin{align*}
	F_{i,j}(x^{(n)})& = \int_0^1f(x_1^{(n)}, \ldots, u(x_i^{(n)} + x_j^{(n)}), \ldots, (1-u)(x_i^{(n)} + x_j^{(n)}), \ldots, x_d)\,du\\
	&=\begin{cases} 
	f(x^{(n)}), \quad\quad \quad x_i^{(n)} + x_j^{(n)} = 0\\
	\frac{1}{x_i^{(n)} + x_j^{(n)}}\displaystyle \int_0^{x_i^{(n)} + x_j^{(n)} }f(x_1^{(n)}, \ldots, s, \ldots, x_i^{(n)} + x_j^{(n)} -s, \ldots, x_d)\,ds, \quad \text{otherwise.}
	\end{cases}
\end{align*}
	Focus on the integral of the second branch for a moment. We discretise the integral on $\Delta_{N-1}(n)$ with $s$-values  $0, 1/n, \ldots, x_i^{(n)} + x_j^{(n)}$. Then 
	\begin{align*}
	\Big|\int_{k/n}^{(k+1)/n} &f(x_1^{(n)}, \ldots, s, \ldots, x_i^{(n)} + x_j^{(n)} -s, \ldots, x_d)\,ds \\
	&\phantom{xxxxxxxxxxx}- n^{-1}f(x_1^{(n)}, \ldots, k/n, \ldots, x_i^{(n)} + x_j^{(n)} - k/n, \ldots, x_d)\Big| < \e/n.
	\end{align*}
Therefore, the overall error, 
\begin{equation}
\left|F_{i,j}(x^{(n)}) - \sum_{k=0}^{n(x_i^{(n)} + x_j^{(n)})} \frac{f(x_1^{(n)}, \ldots, k/n, \ldots, x_i^{(n)} + x_j^{(n)} - k/n, \ldots, x_d)}{n(x_i^{(n)} + x_j^{(n)})}\right| < \e\Big(1+\frac{1}{n}\Big).
\end{equation} 

Now we turn to prove the weak convergence:
\allowdisplaybreaks
\begin{align*}
		\E(f(Y^{(n)}_1)) &= \sum_{x \in \Delta_{N-1}(n)} f(x)\P\{ Y^{(n)}_1 = x\}\\
		&= \sum_{x \in \Delta_{N-1}(n)} \sum_{y \in \Delta_{N-1}(n)}f(x)\P\{ Y^{(n)}_1 = x | Y^{(n)}_0 =y\}\mu_0^{(n)}(y)\\
		&= \sum_{y \in \Delta_{N-1}(n)}\mu_0^{(n)}(y) \sum_{x \in \Delta_{N-1}(n)}f(x)\P\{ X^{(n)}_1 = x | X^{(n)}_0 =y\}\\
		&=  \sum_{y \in \Delta_{N-1}(n)}\mu_0^{(n)}(y) \sum_{i,j:, i\neq j }\quad\sum_{\substack{x_i + x_j = y_i + y_j\\ x_k = y_k}} f(x) \P\{ Y^{(n)}_1 = x | Y^{(n)}_0 =y\}\\
		&=\frac{1}{N(N-1)}  \sum_{y \in \Delta_{N-1}(n)}\mu_0^{(n)}(y) \times\\
		&\sum_{i,j: i\neq j }\quad\sum_{\substack{x_i + x_j = y_i + y_j\\ x_k = y_k}} f(x_1, \ldots x_i, \ldots, x_j, \ldots, x_d) \Big(\frac{1}{n(y_i+y_j)+1} \Big)\\
		&=\frac{1}{N(N-1)}  \sum_{y \in \Delta_{N-1}(n)}\mu_0^{(n)}(y) \\
		&\phantom{xxx}\times  \frac{n(y_i + y_j)}{n(y_i + y_j)+1}\sum_{i,j: i\neq j }\Bigg\{\quad\sum_{k= 0}^{n(y_i + y_j)} f(y_1, \ldots, n^{-1}k, \ldots, y_i+y_j - n^{-1}k, \ldots, y_d) \frac{1}{n(y_i + y_j)}\Bigg\}\\
		&=\frac{1}{N(N-1)}  \sum_{y \in \Delta_{N-1}(n)}\mu_0^{(n)}(y) \times \sum_{i,j: i\neq j } F_{i,j}(y) + O(\e) + O(n^{-1})		\\
		&=\frac{1}{N(N-1)}  \sum_{y \in \Delta_{N-1}(n)}\mu_0^{(n)}(y)\sum_{i,j: i\neq j } F_{i,j}(y)+ O(\e) +  O(n^{-1})\\
		&= \frac{1}{N(N-1)} \sum_{i,j: i\neq j }  \E^{\mu^{(n)}_0}(F_{i,j}) + O(\e) + O(n^{-1}).
\end{align*}
Now let $n \to \infty$ and recall that $F_{i,j}$ is bounded and continuous to conclude that 
\[
- C\e \le  \varliminf_{n\to \infty} \E(f(Y^{(n)}_1)) -   \frac{1}{N(N-1)} \sum_{i,j: i\neq j }  \E^{\mu^{(\infty)}_0}(F_{i,j}) \le \varlimsup_{n\to \infty} \E(f(Y^{(n)}_1)) -   \frac{1}{N(N-1)} \sum_{i,j: i\neq j }  \E^{\mu^{(\infty)}_0}(F_{i,j}) \le C\e.
\]
where $C$ is a constant independent of $n$ that comes from the error term.
Let $\e  \to 0$ to conclude the limit exists and observe that the definition of $F_{i,j}$ and the disintegration theorem imply that 
\[
\lim_{n\to \infty} \E(f(Y^{(n)}_1)) =   \frac{1}{N(N-1)} \sum_{i,j: i\neq j }  \E^{\mu^{(\infty)}_0}(F_{i,j}) = \E(f(X^{(\infty)}_1)).
\]

Therefore we have now shown that the $\mu_1^{(n)} \Longrightarrow \mu_1^{(\infty)}$ if $\mu_0^{(n)} \Longrightarrow \mu_0^{(\infty)}$. An inductive construction and the Markov property are enough to guarantee that all one-dimensional marginals converge.
\end{proof}

\subsection{Irreducibility, uniqueness of the invariant measure and stability}

We can now proceed to study of irreducibility, of the uniqueness of the invariant measure and of the stability for the continuous-space Markov chain.

We begin with a proposition that will simplify the mathematical technicalities associated with general state space discrete time Markov chains.
\begin{proposition}[Duality of coagulation and fragmentation]
\label{prop:invdi}
Let $\mathbf{X} (t)$ denote the coagulation-fragmentation Markov chain defined in Section \ref{sec:DC1}. If $\mathbf{X}(t) \sim U[\Delta_{N-1}]$ then
$\mathbf{X}(t+1) \sim U[\Delta_{N-1}]$, as well.
\end{proposition}
\begin{proof}
See \cite{Bertoin} chapter 2, corollary 2.1, page 77.
\end{proof}
This proposition means that the uniform distribution on the simplex $\Delta_{N-1}$ is an invariant distribution for the coagulation-fragmentation chain. 

What we prove in the sequence is that this is the unique invariant measure and the transition kernels converge to it in the total variation norm.
With this goal in mind we begin with some definitions. 

\begin{definition}[Phi-irreducibility]
Let $(S, \mathcal{B}(S), \phi)$ be a measured Polish space. A discrete time Markov chain $\mathbf{X}$ on $S$ is $\phi$-irreducible 
if and only if for any Borel set $A$ the following implication holds: 
\[ \phi(A) > 0 \Longrightarrow L( u, A) > 0, \quad \text{ for all } u \in S.\]
Above we used notation
\[
L(u, A) = P_u\{ X_n \in A  \text{ for some } n\} = P\{X_n \in A\, \text{for some } n |\,X_0 =u \}. 
\]
\end{definition}
This replaces the notion of irreducibility for discrete Markov chains and mens that the chain is visiting any set of positive measure with positive probability.

The existence of a Foster-Lyapunov function V defined as
\begin{definition}[Foster-Lyapunov function]
For a \textit{petite} set $C$ we can find a function $V \ge 0$ and a $\rho >0$ 
so that for all $x \in S$
\be\label{eq:lya} 
	\int P(x, dy) V(y) \le V(x) - 1 + \rho 1\!\!1_C(x),
\ee
\end{definition}
implies convergence of the kernel $P$ of $\phi$-irreducible, aperiodic chain to a unique equilibrium measure $\pi$
\be 
	\sup_{A \in \mathcal B(S)} \left| P^n(x, A) - \pi(A)\right| \to 0, \text{ as } n\to \infty.
\ee
 (see \cite{Meyn1993}) for all $x$ for which $V(x) < \infty$.  If we define $\tau_C$ to be the number of steps it takes the chain to return to the set $C$, the existence of a Foster-Lyapunov function 
(and therefore convergence to a unique equilibrium)
 is equivalent to $\tau_C$ having finite expectation, i.e.
\[
\sup_{x \in C} \mathbf E_x({\tau_C}) < M_C
\]
which is in turn is implied when $\tau_C$ has geometric tails. This is in fact what we prove is the sequence. 

In our case, $\phi$ will be the Lebesgue measure and the role of the petite set $C$ will be played 
by any set with positive Lebesgue measure. This useful simplification of the mathematical technicalities is an artefact of the 
compact state space ($\Delta_{N-1}$) and the fact that the uniform distribution on the simplex is invariant for the chain (Proposition \ref{prop:invdi}).

\begin{proposition}\label{prop:leb:irr}
	Let $t \in \mathbb N$. The discrete chain $\mathbf{X} = \{X_n\}_{n\in \mathbb N}$ as defined in Section 2.1 is $\phi$-irreducible, where $\phi \equiv \lambda_{N-1}$ is the Lebesgue measure on the simplex.
\end{proposition}
At this point it is useful to explain the idea of the proof of Proposition \ref{prop:leb:irr} when we have deterministic dynamics. We do this in the (easy to visualise) case $N =3$, while the proof is done generally, with Markov dynamics.
For any pair $u, v \in \Delta_2^\circ$, there is a deterministic way to move from $u=(x_u,y_u,z_u)$ to $v=(x_v,y_v,z_v)$ in precisely two steps. The same happens in higher dimensions; on $\Delta_{N-1}^\circ$ we can move from any starting point to any target point using deterministic coagulation-fragmentation dynamics in precisely $N-1$ steps.

	Since the dynamics is symmetric with respect to the coordinates,  we may assume without loss of generality, 
	that  $z_u \le 2/3$ and therefore there exists an entry in $v$, say $x_v$, such that 
	$ m_1 = \displaystyle\frac{x_v}{1-z_u} \le 1.$
	Furthermore, $m_2 = \displaystyle\frac{x_v}{1-y_v} \le 1.$ Then consider the mapping 
	\begin{align}
	u = (x_u, &y_u, z_u) \mapsto (m_1(x_u + y_u), (1-m_1)(x_u + y_u), z_u) \notag \\
	&\mapsto \left(m_1(x_u + y_u), m_2[(1-m_1)(x_u + y_u) + z_u], (1-m_2)[(1-m_1)(x_u + y_u) + z_u]\right) \notag \\
	&= (m_1(1-z_u), m_2[(1-m_1)(1-z_u) + z_u], (1-m_2)[(1-m_1)(1-z_u) + z_u]) \notag\\
	&=(x_v,y_v,z_v) =v. \label{eq:det:map}
	\end{align}
	This idea captures the proof of the Lebesgue - irreducibility (see also Fig. \ref{fig:Delta2}). 

	\begin{figure}[h]
	\begin{center}
		\begin{tikzpicture}[scale=1.3]
		\fill[blue!20](0,0)--(2, 2*1.71)--(3.5,0.5*1.71)--(3,0)--(0,0);
		\draw[dashed](0,0)--(2, 2*1.71)--(3.5,0.5*1.71)--(3,0)--(0,0);
			\draw(0,0)--(4,0)--(2, 2*1.71)--(0,0);
			\draw[nicosred, line width=1.5pt] (0.5,0.5*1.71)--(3.5,0.5*1.71);
			\draw[nicosred, line width=2pt] (1,0)--(2.5,1.5*1.71);
			\draw[ball color= nicosred] (3,0.5*1.71)circle(0.9mm)node[below right]{\small$u$};
			\draw[ball color= nicosred] (1.5,0.5*1.71)circle(1mm)node[below right]{\tiny Step 1};
			\draw[ball color= nicosred] (2,1.71)circle(1mm)node[below right]{\small$v$}; 
			
		\end{tikzpicture}
	\end{center}
	\caption{Schematic of a possible coagulation-fragmentation route from $u$ to $v$ in two steps.  
		Starting from point $u \in \Delta_2$, fix $z_u$. Then on the line $x + y = 1-z_u$, 
		pick the point $(x_v, 1-z_u -x_v, z_u)$. 
		From there, fix $x_v$ and choose $(y_v, z_v)$ on the line $1-x_v =y_v + z_v$. The shaded region are all points $v$ that can be 
		reached with this procedure from $u$, first by fixing $z_u$ and then by fixing $x_v$. 
		Points in the white region can be reached from $u$ first by fixing $z_u$ and then $y_v$.}	
		\label{fig:Delta2}
	\end{figure}
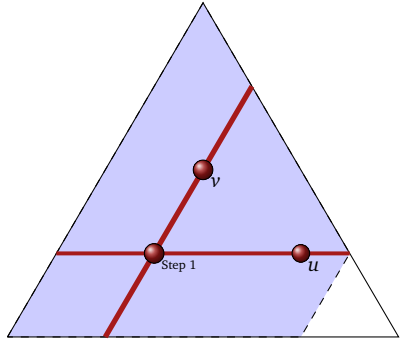

\begin{proof}[Proof of Proposition \ref{prop:leb:irr}] 
	
	First observe that excluding one coordinate (say $x_1$) from the coagulation process in $\Delta_{N-1}$, 
	we are merely restricting the dynamics to  $(1-x_1)\Delta_{N-2}$. This observation is what allows us 
	to proceed by way of induction.

	\textit{Base case $N = 3$.} We choose the base case $N=3$ for purposes of clarity, in a way that can be 
	immediately generalised to higher dimensions. 
	We are working on $(\Delta_2, \mathcal{B}(\Delta_2), \lambda_2 \equiv \lambda \otimes \lambda)$.
	
	 Let $A$ be a Borel set and assume $\lambda_2(A) = \alpha > 0$. 
	We will show that starting from any $x$, the probability of hitting $A$ 
	in just two steps with the coagulation-fragmentation dynamics described above,
	is strictly positive. 
	
	For any $\delta>0$ and point $u$, let $\delta \Delta_2(u)$ denote the scaled simplex 
	with length side $\delta\sqrt{2}$ and barycentre $u$. 
	
	Since $A$ has positive Lebesgue measure, for any $\varepsilon >0$ we can find an open set 
	$G_{A, \varepsilon} \supseteq A$ so that $\lambda_2(G_{A,\varepsilon} \setminus A) < \varepsilon$. Fix an $\varepsilon >0$ and construct $G_{A,\varepsilon}$. 
	Enumerate all rationals in $G_{A, \varepsilon}$ and find $\delta = \delta(A, \varepsilon)> 0$ so that 
	\[ 
	A \subseteq \bigcup_{q \in G_{A, \varepsilon}} ( \delta \Delta_2(q)), \quad \text{and} \quad \lambda_2
	\Big(  \bigcup_{q \in G_{A, \varepsilon}} ( \delta \Delta_2(q))\Big) \le \alpha +2\varepsilon.
	\]
	without loss of generality we may assume that $\delta \Delta_2(q)\cap A \neq \varnothing$ 
	for all $q$ in the union, otherwise we remove the extraneous simplexes from the union. 
	Since the union is countable, there must be a barycentre $q_0$ such that 
	\[ \beta:= \lambda_2(\delta \Delta_2(q_0)\cap A) > 0.\]
	
	Let $u$ be an arbitrary starting point of the process. Without loss of generality, 
	and by possibly decreasing our initial choice of $\delta$, 
	assume
	\begin{enumerate}
	 	\item  $z_u \le 2/3$,
	 	\item $u$ can be deterministically mapped at any point $v \in \delta \Delta_2(q_0)\cap A$ 
		by first fixing $z_u$ and then $x_v$, as in calculation \eqref{eq:det:map}. 
	\end{enumerate} 
	 We denote the three corners $\delta \Delta_2(q_0)$ 
	by $a = (x_a,y_a,z_a), b=(x_a-\delta, y_a+\delta, z_a), c =(x_a-\delta, y_a, z_a+\delta)$. Then,
\allowdisplaybreaks
	\begin{align*}
	0< \beta & = \int_{A \cap\delta \Delta_2(q_0)}\,d\lambda_2
	= \int \int_{A \cap\delta \Delta_2(q_0)} \, d\lambda_1\,d \lambda_1 \\
	& = \int d\lambda_1 \Big(1\!\!1\{ x_a-\delta  <x < x_a  \}  \int d\lambda_1 1\!\!1\{ A\cap\Delta_2(q_0) \cap \{ z + y =1-x  \}\}\Big)\\
	&= \int d\lambda_1 \Big(1\!\!1\{ x_a-\delta  <x < x_a  \}  \int d\lambda_1 1\!\!1\{ (x,y,z) \in A\cap\Delta_2(q_0) : y+z = 1-x\}\Big).
	\end{align*}
	Thus, for a positive $\lambda_1$ measure of $x \in (x_a - \delta, x_a)$ we can find positive $\lambda_1$ measure 
	of the intersection between the set $A$ and the line $z+y=1-x$. Thus, we restrict to the measurable set $F = \{x \in [x_a-\delta, x_a]: \gamma_x > n^{-1}\}$ where
	\[\gamma_x = \lambda_1 \{ A \cap\delta \Delta_2(q_0)\cap \{  y + z =1-x\}\}. \]
	Integer $n$ is chosen large enough so that $\lambda_1(F) > 0$.	
	We have established the existence of a set $C$ so that 
	\[
	C =  \{ (x,y,z):  x \in F, (x,y, z) \in A\cap\delta \Delta_2(q_0)\}, \quad \lambda_2(C)>0.
	\]
	
	This is enough to finally complete the proof of the base case.
	Recall the starting point $u = (x_u, y_u, z_u)$ of the Markov chain. Define the projection set 
	\[ F_{C,u} = \{ (x, y, z):  z = z_u, \exists (x_0, y_0, z_0) \in C  \text{ s.t. } y + z_u = y_0 + z_0 = 1- x_0, x = x_0 \}, \]
	which has positive measure as $\lambda_1(F) = \lambda_1(F_{C,u})$.

	With strictly positive probability, we select to coagulate the first and second coordinate. 
	Then with strictly positive probability, 
	we fragment into the set $F_{C, u}$. 
	This is because the coagulation-fragmentation process of two coordinates picks a uniformly 
	distributed point on the line $x + y = 1-z_u$ by virtue of construction. The uniform distribution is a scalar 
	multiple of the Lebesgue measure, thus guaranteeing that the probability of selecting such 
	a point is strictly positive. 
	Then, given the chain's current position, with strictly positive probability we coagulate the last two coordinates. 
	For the same reason as before, with strictly positive probability we terminate in the set 
	$C\subseteq A$ since for any  fixed $x \in F$, the fragmentation 
	has probability no less than $1/n$ to pick up a point $(x, y_0,z_0) \in C$. 
	
	To conclude, in just two steps we have a positive probability of hitting $A$ from 	any starting point $u$. 
	
	\textit{Induction case:} Now consider simplex $\delta_{N-1}$, when $N \ge 4$ and assume that the proposition
	 is true for all $k < N$. Let $A \subseteq \Delta_{N-1}$ be a Borel set of positive $\lambda_{N-1}$ measure. 
	 As for the base case, we can find a simplex $\delta \Delta_{N-1}(q_0)$ such that 
	 $\lambda_{N-1}(A\cap \delta \Delta_{N-1}(q_0))>0$ and with the same Fubini-Tonelli argument conclude that 
	 there exist a positive integer $n$ and a positive $\lambda_1$-measure set $F$ of $x$ values so that 
	 \[\lambda_{N-2}(A\cap\delta\Delta_{N-1}(q)\cap\{x=x_0\in F\}) > n^{-1}.\]
	 
	 Without loss of generality assume that from the starting point $u$ we can 
	 coagulate and fragment two coordinates, say $u_1$ and $u_2$ so that  
	 $\lambda_1\{x \in F,   x < u_1 + u_2\} > 0$. Then, for the Markov chain, this implies 
	 	\be \label{eq:p1}
		P_u\{ X_1 \cdot e_1 \in F\} > 0. 
		\ee
	 
	 Let 
	 \[ B = \{ X_2, X_3, \ldots, X_{N-1} \;  \text{does not coagulate the first coordinate}\}.\]
	 Again, 
	 \be \label{eq:p2}
		P_u\{ B | X_1 \cdot e_1 \in F\}  = P_u\{ B \} > 0. 
		\ee 
	 Then it is immediate to compute 
	 \begin{align*}
	 	L(u, A) &= P_u\{ X_\ell \in A  \text{ for some } \ell\} \ge P_u\{ X_{N-1} \in A \} \\
		&\ge P_u\{ X_{N-1} \in A, X_1 \cdot e_1 \in F, B\}\\
		&\ge P_u\{ X_{N-1} \in A | B, X_1 \cdot e_1 \in F \} P_u\{ B | X_1 \cdot e_1 \in F \} P_u\{ X_1 \cdot e_1 \in F\}
		> 0.
	 \end{align*}
	 Strict positivity of the last two factors follows from \eqref{eq:p1}, \eqref{eq:p2} while 
	 	$P_u\{ X_{N-1} \in A | B, X_1 \cdot e_1 \in F \}$ equals the probability that the $N-2$ 
		dimensional fragmentation-coagulation process starting from a random point $u_0$ with 
		$x_{u_0} \in F$ hits the set $A\cap \delta \Delta_{N-1}(q)\cap\{ x = x_0\}$ in $N-2$ steps. 
		By the induction hypothesis this probability is strictly positive (given the starting point). 
		By restricting the set $F$ so that its measure remains positive, we may further assume 
		that these probabilities are uniformly bounded away from $0$, independently of the starting point.
		This concludes the proof.
	 \end{proof}

\begin{proposition}[Existence of a Foster-Lyapunov function]
	\label{prop:lyap:ex} 
	The return times $\tau_A$ to any set $A \in \mathcal{B}(\Delta_2)$ of positive measure, have at most geometric tails, under $P_{x_0}$. As a consequence, the Foster-Lyapunov function exists. 
\end{proposition}

\begin{proof}[Proof of Proposition \ref{prop:lyap:ex}]
	Let $A$ be a positive Lebesgue measure set. 
	By repeating the construction in the proof of Proposition \ref{prop:leb:irr} 
	to all coordinates we can find 
	$\alpha_i >0, n_i>0, 1 \le i \le N,  
	\delta>0$ and a rational barycentre $q_0$ and a sequence of measurable sets 
	\[A \supseteq A_1 \supseteq A_2 \supseteq \ldots \supseteq A_N\]
	of positive measure, $\lambda_{N-1}(A_N)=\eta > 0$ 
	and a collection of 1-dimensional measurable sets 
	$F_1, \ldots, F_N$
	with the following properties: 
		\begin{enumerate}
			\item $\lambda_{N-2}\{ A \cap\delta \Delta_2(q_0)\cap \{x_1 = x_1^* \in F_1\}\} =\gamma(x_1^*) > n_1^{-1}, \quad \lambda_1(F_1)\ge \alpha_1$, \\
			$A_1 = \{A \cap\delta \Delta_2(q_0), x_1 \in F_1\}$
			\item $\lambda_{N-2}\{ A_{k-1}\cap\delta \Delta_2(q_0)\cap \{x_k = x_k^* \in F_k\}\} =\gamma(x_k^*) > n_k^{-1}, \quad \lambda_1(F_k)\ge \alpha_k$, \\
			$A_k = \{A_{k-1} \cap\delta \Delta_2(q_0), x_k \in F_k\}, \quad k \ge 2.$
		\end{enumerate} 
		
		The basic property of $A_N$ is that it is accessible with positive probability (that depends on $A$), 
		uniformly bounded from below 
		from any point $u_0 \in \Delta_{N-1}$. 
		Let 
		$a = \min\{a_1,\ldots, a_N,1\}$ and $n_0 = \max\{ n_1,\ldots,n_N\}$. 
		We bound above the probability that we do not hit $A_N$ in the first $N-1$ steps, i.e. 
		$P_{u_0}\{ \tau_{A_N}> N-1\}$. 
		 Suppose we hit in $N-1$ steps or less. Then there is at least one sequence 
		 of $N-1$ coagulation-fragmentation steps for which, 
	if we follow it we land in $A_N$. We select the appropriate pair of indices at each step 
	with probability $1/N(N-1)$ and, given this, 
	we fragment at an appropriate point with probability at least $a$. Therefore 
	\begin{align*}
	\inf_{x \in \Delta_{N-1}}P_x\{ \tau_{A_N} \le N-1\} 
	&
	\ge  \left(\frac{2a}{N(N-1)}\right)^{N-1}
	=\rho_A >0.
	\end{align*}
	Pick a starting point $u_0 \in A$. Then $P_{u_0}( \tau_A > M) \le P_{u_0}( \tau_{A_N} > M)$. 
	We will show that the larger tail is bounded above geometrically by an expression independent of $u_0$.
	We compute  
	\begin{align*}
	P_x\{ \tau_{A_N} > (N-1)M\} &= P_x\{ X_1 \notin A_N,\ldots, X_{(N-1)M} \notin A_N\} \\
		&\le \left(\sup_{u \in \Delta_{N-1}\setminus A_N} P_u\{X_1 \notin A_N,\ldots, X_{N-1} \notin A_N \}\right)^M  \\
		&\le \left(\sup_{u \in \Delta_{N-1}\setminus A_N} P_u\{\tau_{A_N} > N-1\} \right)^M  \\
		&\le (1-\rho_A)^M.
	\end{align*}
	Finally,  since for any $1\le k \le N-1$ we have $ \{ \tau _{A_N} > (N-1)(M+1) \}\subseteq \{ \tau _{A_N} > (N-1)M+ k \} \subseteq \{ \tau_{A_N} > (N-1)M\}$ we conclude that $\tau_{A_N}$ has geometric tails.
	\end{proof}

We assemble these propositions in the following theorem.

\begin{theorem}
	 Let $\mathbf{X} (t)$ denote the coagulation-fragmentation Markov chain defined in Section \ref{sec:DC1} and initial distribution $\mu_0$ on $\Delta_{N-1}$. 
	 Let $\mu_t$ denote the distribution of  $\mathbf{X} (t)$ at time $t \in \N_0$. Then the  uniform  distribution on $\Delta_{N-1}$ is the  unique invariant distribution that can be found as the weak limit of the sequence 
	 $\mu_t$.
\end{theorem}

\begin{proof} 
From Proposition \ref{prop:invdi} we have that $U[\Delta_{N-1}]$ is an invariant distribution for the process. Since the chain is $\phi$-irreducible as shown in Proposition \ref{prop:leb:irr}, 
uniqueness of the equilibrium follows from the existence of a Foster-Lyapunov 
function, proven in  Proposition \ref{prop:lyap:ex}. 
\end{proof}

\subsection{Kinetic equation as limit of the agent system}

Kinetic equations for the one-agent distribution function with a bilinear
interaction term can be derived using mathematical techniques from the
kinetic theory of rarefied gases \cite{Cerc88,CIP94}.
In this section we discuss how in a time-continuous setting, where the stock (or wealth) of each agent is
a continuous variable $w\in\mathcal{I}=[0,\infty)$, the exchange mechanism
discussed above constitutes a special case of a 
kinetic model for wealth distribution, proposed by Cordier, Pareschi,
Toscani in \cite{CPT05}. In this setting the microscopic dynamics lead to a
    homogeneous Boltzmann-type equation for the distribution function
    of wealth $f=f(w,t)$. One can study the moment evolution of the
    Boltzmann equation to obtain insight into the tail behaviour of
    the cumulative wealth distribution.
We also discuss the grazing  collisions limit which yields a
macroscopic Fokker-Planck-type equation. 

Cordier, Pareschi, Toscani \cite{CPT05} propose a kinetic model for wealth
distribution where wealth is exchanged between individuals through
pairwise (binary) interactions: when two individuals with
pre-interaction wealth $v$ and $w$ 
meet, then their post-trade wealths $v^*$ and $w^*$ are given by
\begin{align}\label{e:exchange}
v^* =(1-\lambda) v +\lambda w + \tilde\eta v,\quad
w^* = (1-\lambda) w + \lambda v + \eta w. 
\end{align}
Herein, $\lambda\in (0,1)$ is a constant, the so-called {\em propensity to invest\/}.
The quantities $\tilde\eta$ and $\eta$ are independent random variables
with the same distribution (usually with mean zero and finite variance
$\sigma^2$). They model randomness in the outcome of the interaction
in a diffusive fashion. 
Note that to ensure that post-interaction wealths remain in the interval
$\mathcal{I}=[0,\infty)$ additional assumptions need to be made.
The discrete exchange dynamics considered in the
previous sections find their continuous kinetic analogue when setting
$\eta=\tilde\eta\equiv 0$ in \eqref{e:exchange}.

With a fixed number $N$ of agents, the interaction \eqref{e:exchange}
induces a discrete-time Markov process on $\mathbb{R}_+^N$ with
$N$-particle joint probability distribution
$P_N(w_1,w_2,\dots,w_N,\tau)$.
One can write a kinetic equation for the one-marginal distribution
function
$$
P_1(w,\tau)=\int P_N(w,w_2,\dots,w_N,\tau)\, dw_2 \cdots dw_N,
$$
using only one- and two-particle distribution functions \cite{Cerc88,CIP94},
\begin{multline*}
P_1(w,\tau+1)-P_1(w,\tau)=\\
\Bigg\langle \frac 1N \Biggl[\int
P_2(w_i,w_j,\tau)\bigl( \delta_0(w-((1-\lambda) w_i +\lambda w_j +
\tilde\eta w_i))+\delta_0(w-((1-\lambda) w_j + \lambda w_i + \eta
w_j)) \bigr)\, dw_i\, dw_j - 2P_1(w,\tau) \Biggr ]\Biggr\rangle.
\end{multline*}
Here, $\langle\cdot\rangle$ denotes the mean operation with respect to
the random variables $\eta,\tilde\eta$. 
This process can be continued to give a hierarchy of equations of so-called
BBGKY-type \cite{Cerc88,CIP94}, describing the dynamics of the system of a
large number of interacting agents.
 A standard approximation is to
neglect correlations between the wealth of agents and assume the
factorization
$$
P_2(w_i,w_j,\tau)=P_1(w_i,\tau)P_1(w_j.\tau).
$$
Standard methods of kinetic theory \cite{Cerc88,CIP94} can be used to
show that, scaling time as $t=2\tau/N$ and taking the thermodynamical
limit $N\to\infty$, one obtains that the time-evolution of the
one-agent distribution function is governed by a
homogeneous Boltzmann-type equation of the form
\begin{multline}
  \label{e:boltzmann}
\frac{\partial }{\partial t}f(w,t) = \\
\frac 12 \Biggl\langle\int
f(w_i,t)f(w_j,t) \bigl( \delta_0(w-((1-\lambda) w_i +\lambda w_j +
\tilde\eta w_i))+\delta_0(w-((1-\lambda) w_j + \lambda w_i + \eta
w_j)) \bigr) \, dw_i\, dw_j \Biggr\rangle - f(w,t).
\end{multline}

Recalling the results from \cite{DuMaTo08,MaTo08}, we have the
following proposition.
\begin{proposition}
The distribution $f(w,t)$ tends to a steady state distribution $f_\infty(w)$ with
an exponential tail.
\end{proposition}
\begin{proof}
 The results in \cite{DuMaTo08,MaTo08} imply that $f(w,t)$
tends to a steady distribution $f_\infty(w)$  which depends on the initial distribution only through the conserved mean wealth $M =\int_0^\infty w\,f(w,t)\,dw >0$.
As detailed in \cite{DuMaTo08,MaTo08}, the long-time behaviour of the $s$-th
moment $\int_0^\infty w^s\,f(w,t)\,dw $ is characterised by the function
$  \mathcal{S}(s) =(1-\lambda)^s+\lambda^s- 1$ which is negative for
all $s>1$, hence all $s$-th moments for $s>1$ are bounded, and the tail
of the steady state distribution is exponential.\qed
\end{proof}

In general, such equations like \eqref{e:boltzmann} are rather difficult to treat and it is usual in kinetic
theory to study certain asymptotic limits.
In a suitable scaling limit, a partial differential
equation of Fokker-Planck type can be derived for the distribution of
wealth. Similar diffusion equations are also obtained
in \cite{SlaLav03} as a mean field limit of the Sznajd model \cite{SznSzn00}.
Mathematically, the model is related to works in the
kinetic theory of granular gases \cite{CIP94}.

To this end, we study by formal asymptotics 
the so-called continuous trading
limit ($\lambda \to 0$ while keeping $\sigma_\eta^2/\lambda = \gamma$
fixed).

Let us introduce some notation. First, consider test-functions
$\phi \in\mathcal{C}^{2,\delta}([0,\infty))$ for some $\delta>0$. We use
the usual H\"older norms 
\begin{align*}
 \|\phi\|_{\delta} = \sum_{|\alpha|\leq 2}\|D^\alpha\phi\|_{\mathcal{C}}+\sum_{\alpha=2}[D^\alpha\phi]_{\mathcal{C}^{0,\delta}},
\end{align*}
where $[h]_{\mathcal{C}^{0,\delta}} = \sup_{v\neq w}{|h(v)-h(w)|}/{|v-w|^\delta}.$
Denoting by 
$\mathcal{M}_0(A)$, $A\subset\mathbb{R},$ the space of
probability measures on $A$, we define by
$$ 
\mathcal{M}_p(A)=\bigg\{\Theta\in\mathcal{M}_0\;\bigg|\;\int_A
|\eta|^pd\Theta(\eta)<\infty,\,p\geq 0\bigg.\bigg\}
$$
the space of measures with finite $p$th moment. In the following all our probability densities belong to $\mathcal{M}_{2+\delta}$ and we assume that the density $\Theta$ is obtained from a random variable $Y$ with zero mean and unit variance. We then obtain
\begin{align}\label{e:expvalue}
 \int_{\mathcal{I}}|\eta|^p\Theta(\eta)\;d\eta = {\mathrm E}[|\sigma_\eta  Y|^p]=\sigma_\eta^p{\mathrm E}[|Y|^p],
\end{align}
where ${\mathrm E}[|Y|^p]$ is finite.
The weak form of \eqref{e:boltzmann} is given by
\begin{equation}
\label{weakform}
\frac{d}{dt}\int_{\mathcal{I}} f(w,t) \phi(w)\,dw
=\int_{\mathcal{I}}
\mathcal{Q}(f,f)(w)\phi(w)\,dw
\end{equation}
where 
\begin{align*}
\begin{split}
\int_{\mathcal{I}}&\mathcal{Q}(f,f)(w)\phi(w)\,dw {}\\
&=\frac12\Big\langle \int_{\mathcal{I}^2}  \bigl(  \phi(w^*)+\phi(v^*)-\phi(w)-\phi(v)\bigr)
f(v) f(w) \,dv \, dw\Big\rangle .
\end{split}
\end{align*}
Here, $\langle\cdot\rangle$ denotes the mean operation with respect to
the random variables $\eta,\tilde\eta$.
To study the situation for large times, i.e.\ close to the steady
state, we introduce for $\lambda\ll 1$ the transformation
$\tilde t=\lambda t, \; g(w,\tilde t )=f(w,t).$
This implies $f(w,0)=g(w,0)$ and the evolution of the scaled
density $g(w,\tilde t)$ follows (we immediately drop the tilde in the
following)
\begin{equation}
\label{weakform2}
\frac{d}{dt}\int_{\mathcal{I}} g(w,t)
\phi(w)\,dw=\\
\frac1{\lambda}\int_{\mathcal{I}}
\mathcal{Q}(g,g)(w)\phi(w)\,dw.
\end{equation}
Due to the interaction rule \eqref{e:exchange}, it holds
\begin{align*}
w^*-w = \lambda (v-w) + \eta w.
\end{align*}
A Taylor expansion of $\phi$ up to second order around $w$ of the right
hand side of \eqref{weakform2} leads to
\begin{align*}
&\Big\langle \frac{1}{\lambda}\int_{\mathcal{I}^2 }\phi'(w)\left[\lambda (v-w) + \eta w\right]g(w)g(v)\,dv\,dw\Big\rangle\\
&\phantom{xxxxxxx}+\Big\langle \frac{1}{2\lambda}\int_{\mathcal{I}^2}\phi''(\tilde{w})\left[\lambda (v-w) + \eta w\right]^2g(w)g(v)\,dv\,dw\Big\rangle\\
=&\Big \langle\frac{1}{\lambda}\int_{\mathcal{I}^2 }\phi'(w)\left[\lambda (v-w) + \eta w\right]g(w)g(v)\,dv\,dw\Big \rangle\\
&\phantom{xxxxxxx}+\Big \langle \frac{1}{2\lambda}\int_{\mathcal{I}^2 }\phi''({w})\big[\lambda (v-w) + \eta w\big]^2g(w)g(v)\,dv\,dw\Big \rangle +R(\lambda,\sigma_\eta)\\
=&-\int_{\mathcal{I}^2}\phi'(w)(w-v)g(w) g(v)\,dv\,dw\\
&\phantom{xxxxxxx}+ \frac{1}{2\lambda}\int_{\mathcal{I}^2}\phi''(w)\big[\lambda^2
(v-w)^2+\lambda\gamma w^2\big]g(w)g(v)\,dv\,dw+R(\lambda,\sigma_\eta),
\end{align*}
with $\tilde{w}=\kappa w^*+(1-\kappa)w$ for some $\kappa\in [0,1]$ and
\begin{equation*}
R(\lambda,\sigma_\eta)=\Big\langle
\frac{1}{2\lambda}\int_{\mathcal{I}^2}(\phi''(\tilde{w})-\phi''(w))  \left[\lambda (v-w) + \eta w\right]^2g(w)g(v)\,dv\,dw\Big\rangle.
\end{equation*}
Now we consider the limit $\lambda, \sigma_\eta  \to 0$ while keeping $\gamma=\sigma_\eta^2/\lambda$
fixed. 
It can be seen that the remainder term $R(\gamma,\sigma_\eta)$
vanishes in this limit, see \cite{CPT05} for details. 
In the same limit, the term on the right
hand side of \eqref{weakform2} converges to
\begin{multline*}
-\int_{\mathcal{I}^2}\phi'(w)(w-v)g(w)g(v)\,dv\,dw
+\frac{1}{2}\int_{\mathcal{I}^2}\phi''(w)\gamma
  w^2g(w)g(v)\,dv\,dw\\
=-\int_{\mathcal{I}}\phi'(w)(w-m)g(w)\,dw
+\frac{\gamma}{2}\int_{\mathcal{I}}\phi''(w)
w^2g(w)\,dw,
\end{multline*}
with $m = \int_{\mathcal{I}} v g(v) \,dv$ being the mean wealth (the mass
is set to one for simplicity, otherwise it would appear as well here).
After integration by parts we obtain the right hand side of (the weak form of) the Fokker-Planck equation
\begin{equation}
\label{e:fokkerplanck}
\frac{\partial}{\partial t}g(w,t)=\frac{\partial}{\partial w}\Big((w-m)g(w,t)\Big){}
+\frac{\gamma}{2}\frac{\partial^2}{\partial w^2}\Big(w^2g(w,t) \Big),
\end{equation}
subject to no flux boundary conditions (which result from the
integration by parts). The same equation has also been obtained by
considering the mean-field limit in a trading model described by
stochastic differential equations \cite{BouMez00}.

\section{Remembering Jun-ichi Inoue}

One of us (Enrico Scalas) was expecting to meet Jun-ichi Inoue at the 2015 AMMCS-CAIMS Congress in Waterloo, Ontario, Canada. Together with Bertram D\"uring (also co-author of this paper), we organised a special session entitled {\em Wealth distribution and statistical equilibrium in economics} (see: {\tt http://www.ammcs-caims2015.wlu.\-ca/special-sessions/wdsee/}). Even if Enrico never collaborated with Jun-ichi on the specific problem discussed in this paper, they co-authored two research papers, one on the non-stationary behavior of financial markets \cite{Jun-ichi1} and another one on durations and the distribution of first passage times in the FOREX market \cite{Jun-ichi2}. The former was the outcome of a visit of Jun-ichi to the Basque Center for Applied Mathematics in Bilbao from 3 October 2011 to 7 October 2011. Enrico, Jun-ichi and Giacomo Livan met several times in front of blackboards and computers and the main idea of the paper (nonstationarity of financial data) was suggested by Jun-ichi. The latter is the result of a collaboration with Naoya Sazuka who, among other things, provided the data from Sony Bank FOREX transactions. This paper is connected to Enrico's activity on modelling high-frequency financial data with continuous-time random walks. A third review paper was published on the role of the inspection paradox in finance \cite{Jun-ichi3}.  Before leaving for Canada, Enrico received the sad news of Jun-ichi's death. He had the time to change his presentation in Waterloo to include a short commemoration of Jun-ichi. With Jun-ichi, Enrico lost not only a collaborator, but a friend with an inquisitive mind.

\section*{Acknowledgments}
BD acknowledges support by the Leverhulme Trust research project grant
“Novel discretisations for higher-order nonlinear PDE”
(RPG-2015-69). NG acknowledges support by the EPSRC grant
EP/P021409/1. We would also like to thank Fei Cao for noting an error
in a previous version of this paper which is now corrected. 




\end{document}